\documentclass[12pt,a4paper]{article}
\usepackage[T2A]{fontenc}

\usepackage[utf8]{inputenc}
\usepackage[english]{babel}

\usepackage{amsfonts,latexsym,amssymb,amsthm,amsmath,euscript,cite} 
\usepackage[a4paper,left=2cm,right=2cm,top=2cm,bottom=1.85cm]{geometry}

\usepackage[OT4]{fontenc}

\usepackage{mathtools}
\usepackage[T1]{fontenc}
\usepackage[utf8]{inputenc}
\usepackage{hyperref}
\usepackage{url}
\usepackage[T1]{fontenc}
\usepackage{lmodern}
\usepackage{amsfonts}
\usepackage{graphicx}
\usepackage{color}

\DeclareTextCompositeCommand{\k}{LY1}{e}
{\oalign{e\crcr\noalign{\kern-.27ex}\hidewidth\char7\hidewidth}}

\numberwithin{equation}{section}
\theoremstyle{plain}
\newtheorem{theorem}{Theorem}[section]
\newtheorem{lemma}{Lemma}[section]

\newtheorem{remark}{Remark}[section]
\newtheorem{definition}{Definition}[section]

\newtheorem*{corollary*}{Corollary}

\def\col{\mathrm{col}\,}

\def\tr {\mathrm{tr}\,}
\date{}
\begin{document}

\title{The Cauchy problem for the (2+1) integrable nonlinear Schrödinger equation}

\author{ L.~P.~Nizhnik\footnote{Institute of Mathematics NAS of Ukraine, Kyiv, Ukraine, nizhnik@imath.kiev.ua }}
\maketitle

\begin{abstract}
We study the Cauchy problem for the (2+1) integrable nonlinear Schrödinger equation by the inverse scattering transform (IST) method.  This Cauchy problem with given initial data and boundary data at infinity is reduced by IST to the Cauchy problem for the linear Schrödinger equation, in which the potential is expressed in terms of boundary data.  The results on direct and inverse scattering problems for a two-dimensional Dirac system with special potentials are used and refined.  The Cauchy problem admits an explicit solution if the IST of the solution is an integral operator of rank 1. We give  one such solution.
\end{abstract}
\section{\bf{Introduction}}
We consider the (2+1) nonlinear evolutionary Schrödinger equation of the form
\begin{equation}\label{1.1}
   i\frac{\partial u}{\partial t}+\frac{\partial^2 u}{\partial x^2}+\frac{\partial^2 u}{\partial y^2}+(v_1 +v_2)u=0.
\end{equation}
The pseudopotentials $v_1$ and $v_2$ are real functions of $(x,y;t)$ variables and are related to the solution $u(x,y; t)$ of the equation (\ref{1.1}) by the stationary equations
\begin{equation}\label{1.2}
   \frac{\partial v_1}{\partial x}=2\frac{\partial}{\partial y}|u|^2,\qquad \frac{\partial v_2}{\partial y}=2\frac{ \partial}{\partial x}|u|^2.
\end{equation}
Systems of equations close to the reduced one (\ref{1.1})-(\ref{1.2}) are found in applications, and in the theory of surface waves they are called the Davey-Stewartson equations \cite{3,4,5}. The IST-integrable Davey-Stewartson equation is not symmetric with respect to $x$ and $y$, unlike the system (\ref{1.1})--(\ref{1.2}). We will not discuss the applied needs of the system (\ref{1.1})--(\ref{1.2}) here, but consider in detail how to mathematically correctly set the Cauchy problem for such a system and justify the possibility of studying it using the IST method.
The construction of IST for the system (\ref{1.1})-(\ref{1.2}) starts from the Lax representation of such a system in the form of an operator identity
\begin{equation}\label{1.3}
   LP-QL=0,
\end{equation}
where $L,\,P,\,Q$ are differential operators whose coefficients are explicitly expressed in terms of the functions $u,$\,$v_1,$\,$v_2$. Such operators can be
\begin{equation}\label{1.4}
  L= \begin{pmatrix}\frac{\partial}{\partial x}& u\\-\bar{u} & \frac{\partial}{\partial y}\end{pmatrix},\,\, P= \begin{pmatrix} \mathfrak{D}+v_1 & -2u_x\\ -2\bar{u}_y & \mathfrak{D}-v_2\end{pmatrix},\,\,
  Q= \begin{pmatrix} \mathfrak{D}+v_1 & -2u_y\\ -2\bar{u}_x & \mathfrak{D}-v_2\end{pmatrix},
\end{equation}
where
$$ \mathfrak{D}=i\frac{\partial}{\partial t}-\frac{\partial^2}{\partial x^2}+\frac{\partial^2}{\partial y^ 2}.
$$
Substituting (\ref{1.4}) into (\ref{1.3}) results in the system (\ref{1.1})-(\ref{1.2}). Of course, this requires the local existence of all arising derivatives of the functions $u,$\,$v_1,$\,$v_2.$ the scattering problem for which is well studied \cite{1,2} under the assumption $u\in L_2(R^2)$, then the solution of the Cauchy problem for the system (\ref{1.1})--(\ref{1.2}) in this work will be studied for the case of belonging
$u(x,y;t) \in L_2(R^2)$.
Of course, the Cauchy problem for the system (\ref{1.1})--(\ref{1.2}) requires initial conditions $u(x,y;t)|_{t=0}=u_0(x,y) $. It is necessary to specify the functional space for the solution and the initial data. However, the presence of stationary equations relating the pseudopotentials $v_1$ and $v_2$ with the solution $u$ requires a different functional space for pseudopotentials and additional boundary data at infinity. If there are no additional data to the initial data, then there can be many solutions. Thus, if $u(x,y;t),$\,$v_1(x,y;t),$\,$v_2(x,y;t)$ the solution of the system (\ref{1.1})-( \ref{1.2}), then $\hat{u}(x,y;t)=e^{itk}u(x,y;t)$ \,$\hat{v}_1=v_1+k, $\,$\hat{v}_2=v_2$ for any real value of $k$ is also a solution to the system (\ref{1.1})-(\ref{1.2}) with the same initial data $\hat{u} (x,y;t)|_{t=0}=u(x,y;0).$ We also assume that the pseudopotentials $v_1(x,y;t),$\,$v_2(x,y ;t)$ have limit values
   $$v_1(\pm \infty,y;t)=p_{\pm}(y,t),\qquad v_2(x,\pm \infty;t)=q_{\pm}(x,t).
   $$
   In other words, we have
   \begin{equation}\label{1.5}
   v_1(x,y;t)=p_-(y,t)+2 \int_{-\infty}^{x}\,\frac{\partial}{\partial y}|u(s,y;t |^2\,ds,\quad v_2(x,y;t)=q_+(x,t)-2\int_{y}^{+\infty}\,\frac{\partial}{\partial x }|u(x,s;t|^2\,ds.
\end{equation}
Instead of these equalities, we can assume that they hold similar
\begin{equation}\label{1.6}
    v_1(x,y;t)=p_+(y,t)-2 \int_{x}^{+\infty}\,\frac{\partial}{\partial y}|u(s,y;t |^2\,ds,\quad v_2(x,y;t)=q_-(x,t)+2\int_{-\infty}^{y}\,\frac{\partial}{\partial x }|u(x,s;t|^2\,ds,
\end{equation}
if we take other scattering data in IST.
Additional conditions on $p_{\pm},$\,$q_{\pm}$ will be given in the corresponding theorems. Now we give the definition of the Cauchy problem for the system (\ref{1.1})-(\ref{1.2}).
\begin{definition}
The Cauchy problem for the system (\ref{1.1})-(\ref{1.2}) is to find solutions to these equations that satisfy the given initial conditions
$u(x,y;0)=u_0(x,y)\in L_2(R^2)$ and equations (\ref{1.5}) for given $p_{-}(y,t)$ and $q_ {+}(x,t)$ , or
equation (\ref{1.6}) for given $p_{+}(y,t)$ and $q_{-}(x,t)$.
\end{definition}
\begin{remark}
Note that in \cite{6} for the DSII equation, the homogeneous conditions at infinity are replaced by the requirement that the pseudopotential belongs to the space $L_2(R^2)$ with an additional narrowing of the space of solutions of the nonlinear equation. This is equivalent to the condition that the pseudopotential $w(|u|^2)$ depends only on $|u|^2$ and for $u\rightarrow 0$ in some norm, then $w(u_n)\rightarrow 0$ in another.  Roughly speaking, the condition $w(0)=0$ is equivalent to homogeneous boundary conditions. This  follows obviously from (\ref{1.5}) for the equation (\ref{1.1}).
\end{remark}
Note by $\mathfrak{A}$ IST, which takes the potential $u(x,y)\in L_2(R^2)$ in the Dirac operator
$$
\begin{pmatrix}\displaystyle \frac{\partial}{\partial x}& \displaystyle u\\\displaystyle -\bar{u} & \displaystyle \frac{\partial}{\partial y}\end{pmatrix}
$$
into the function $f(\xi,\eta)\in L_2(R^2)$ the scattering data
\begin{equation}\label{1.7}
   \mathfrak{A} u=f.
\end{equation}
The operator $\mathfrak{A}$ is noninear, and its action is constructive and requires the solution of linear Fredholm equations with the Hilbert-Schmidt kernels.
The operator $\mathfrak{A}$ has an inverse $\mathfrak{A}^{-1},$ in other words, if $\mathfrak{A}u_1=f$ and $\mathfrak{A}u_2=f$ then $u_1=u_2=\mathfrak{A}^{-1}f.$
The range of the operator $\mathfrak{A}$ is the set of all functions
$f(\xi,\eta)\in L_2(R^2)$ such that the integral operator $F$ with kernel $f(\xi,\eta)$ in $L_2(R^1)$ has a strictly norm less than 1. The constructive and explicit form of the operator $\mathfrak{A}^{-1}$ will be presented in section 2.
The main property of $\mathfrak{A}$ transform is that the operator $\mathfrak{A}$ maps the solution of the Cauchy problem (\ref{1.1})-(\ref{1.5}) $u(x,y;t)$  into the function $f(\xi,\eta;t)$, which is a solution of the linear evolutionary Schrödinger equation
\begin{equation}\label{1.8}
   i \frac{\partial \overline{ f}}{\partial t}+\frac{\partial^2 \overline{f}}{\partial \xi^2}+\frac{\partial^2 \overline{ f}}{\partial \eta^2}+(p_+(\eta,t)+q_-(\xi,t)) \overline{f}=0
\end{equation}
with the initial condition $f(\xi,\eta;0)=\mathfrak{A}u(x,y;0).$ Therefore, the possibility of using IST in studying the Cauchy problem for the system (\ref{1.1})-(\ref{1.5}) is reduced to whether the time evolution of $f(\xi,\eta;t)$ according to equation (\ref{1.8}) will leave it in the range of transform $\mathfrak{A}$? In other words, will the solution $f(\xi,\eta;t)$ of the Cauchy problem (\ref{1.8}) generate an integral operator $F(t)$ with the kernel $f(\xi,\eta;t)$, whose norm for $t\geq0$ is strictly less than 1? The answer to this question is positive, moreover, the conditions on the boundary data $p_{+}(y,t)$ and $q_{-}(x,t)$ when $||F(t)||=|| F(0)||<1.$ In the case, when in the Cauchy problem for the system
(\ref{1.1})-(\ref{1.5}) or (\ref{1.1})-(\ref{1.6}) the boundary conditions are homogeneous, i.e. $p_{+}\equiv 0$ and $q_{ -} \equiv 0$ or $p_{-}\equiv 0$ and $q_{+} \equiv 0$,
IST is an efficient way to study the Cauchy problem for the system(\ref{1.1})-(\ref{1.2}). In this case, the evolution equation (\ref{1.8}) is a linear differential equation with constant coefficients. Section 5 gives an example of an explicit solution of the Cauchy problem for the system (\ref{1.1})-(\ref{1.5}).

\section{Direct and inverse scattering problems for two-dimensional Dirac systems}
\subsection{The case of arbitrary potentials}
Let us present the well-known results \cite{1,2 } on the scattering problem for Dirac systems
\begin{equation}\label{2.1}
\begin{array}{c}
\displaystyle \frac{\partial \psi_1(x,y)}{\partial x}+ u_1(x,y)\psi_2(x,y)=0, \\
\displaystyle \frac{\partial \psi_2(x,y)}{\partial x}+ u_2(x,y)\psi_1(x,y)=0.
\end{array}
\end{equation}
The potentials in the system (\ref{2.1}) are the complex-valued functions $u_1,\,u_2 \in L_2(R^2)$. For the system (\ref{2.1}), there are admissible solutions with the following asymptotics
\begin{equation}\label{2.2}
\begin{array}{c}
\psi_1(x,y)=a_1(y)+o(1),\,\, x\rightarrow - \infty;\quad \psi_1(x,y)=b_1(y)+o(1),\quad x\rightarrow + \infty; \\
\psi_2(x,y)=a_2(x)+o(1),\,\, y\rightarrow - \infty;\quad \psi_2(x,y)=b_2(x)+o(1),\quad y\rightarrow + \infty,
\end{array}
\end{equation}
where all functions $a_k,$\,$b_k \in L_2(R^1),$ in this case, if one of the pairs of functions from $L_2(R^1)$ \,$(a_1,a_2),$\, $(b_1,b_2),$\,$(a_1,b_2),$\,$(b_1,a_2)$ are fixed, then all the others are uniquely determined through the chosen pair. This means that there are bounded linear operators $b=Sa$ and $a=S^{-1}a$, where $a=\col(a_1,a_2)$ and $b=\col(b_1,b_2)$ . The operator $S$ is called the scattering operator for the system (\ref{2.1}), $a$ are incident waves, $b$ are scattered waves. Moreover, the matrix $2\times2$ operators $S$ and $S^{-1}$ have the form $S=I+F$,\,$S^{-1}=I+G,$ where the matrix operators
$$F=\begin{Vmatrix}
         F_{11} & F_{12} \\
        F_{21} & F_{22}
       \end{Vmatrix}, \quad
       G=\begin{Vmatrix}
         G_{11} & G_{12} \\
        G_{21} & G_{22}
       \end{Vmatrix}$$
       consist of integral operators with the Hilbert-Schmidt kernels.
Wherein
$$(F_{kk}a)(\xi)=\int_{-\infty}^{\xi}\,F_{kk}(\xi,\eta)a(\eta)\,d\eta, \quad (G_{kk}b)(\eta)=\int_{\eta}^{+\infty}\,G_{kk}(\xi,\eta)b(\xi)\,d\xi,
$$
that is, the diagonal elements $F_{11}$ and $F_{22}$ are Volterra integral operators with a variable upper limit of integration, and $G_{kk}$,\,$k=1,2$ are Volterra integral operators with variable lower limit of integration. The operators $S=I+F$,\,$S^{-1}=I+G$ themselves allow two-sided factorization into Volterra factors $S=(I+K_+)(I+K_-)=(I+W_ -)(I+W_+).$ The scattering data for the system (\ref{2.1}) are pairs of integral operators $(F_{12},G_{21})$ or $(F_{21},G_{12} )$ or pairs of their kernels. The main result on the inverse scattering problem for the system (\ref{2.1}) is the unique effective connection of the scattering data with the pair of potentials $(u_1,u_2)$ in the system (\ref{2.1}) and the description of the scattering data \cite{1 ,2}.
\begin{remark}\label{2.1}
Note that due to the equalities $SS^{-1}=S^{-1}S=I$ and the representations $S=I+F$,\,$S^{-1}=I+G$ we have the identities
\begin{equation}\label{2.3}
\begin{array}{c}
  I-F_{12}G_{21}=(I+F_{11})(I+G_{11}), \\
    I-F_{21}G_{12}=(I+F_{22})(I+G_{22}).
\end{array}
\end{equation}
They show that the operators constructed from the scatter data allow either left or right factorization. This property is decisive for a pair of integral operators to be scattering data for the Dirac system with potentials $u_1,\,u_2 \in L_2(R^2)$.
\end{remark}
\begin{remark}\label{2.2}
Let us present here more criteria for the factorizability of the Fredholm operator $I+K$, where $K$ is an integral operator with the Hilbert-Schmidt kernel. For the operator $I+K$ to admit right factorization into Volterra factors, it is necessary and sufficient that for any real $-\infty<\lambda \le \infty$ there exists an operator $(I+Q_{\lambda}KQ_{ \lambda})^{-1}$. There will be a left factorization if and only if there exists an operator $(I+P_{\lambda}KP_{\lambda})^{-1}$. Here $Q_{\lambda}$ and $P_{\lambda}$ are projectors
$$Q_{\lambda}\varphi(x)=\theta(\lambda-x)\varphi(x),\\
P_{\lambda}\varphi(x)=\theta(x-\lambda)\varphi(x).
$$
\end{remark}
Let us present an algorithm for reconstructing the potentials $u_1$ and $u_2$ in the two-dimensional Dirac system (\ref{2.1}) from the known scattering data. Let the scattering data $(F_{2,1}, G_{1,2})$ for the system (\ref{2.1}) be known.
For convenience, we omit the indices and consider the pair as the Hilbert-Schmidt operator $(F,G)$ as scattering data. We will assume that the kernels of integral operators depend on the variables $\xi$ and $\eta$.
Consider two Hilbert--Schmidt operators depending on the parameters
$\lambda$ and $\mu:$
\begin{equation}\label{2.a}
   U_1=[I-GQ_{\lambda}FP_{\mu}]^{-1}\cdot G, \quad U_2=-[I-FP_{\mu}GQ_{\lambda}]^{-1}\cdot F,
\end{equation}
If $(F,G)$ are scattering data, then these operators exist and are the Hilbert--Schmidt integral operators. Their kernels
$$U_1(\xi,\eta)=U_1(\xi,\eta;\lambda,\mu),\quad U_2(\xi,\eta)=U_2(\xi,\eta;\lambda,\mu )
$$
depend on parameters
$\lambda$ and $\mu$. Thus the following equalities are valid
\begin{equation}\label{2.2a}
   u_1(x,y)=U_1(y,x;x,y), \quad u_2(x,y)=U_2(x,y;x,y).
\end{equation}

We can obtaine the formula (\ref{2.2a}) by solving a system of solvable linear integral equations \cite{1,2}. Therefore, the above algorithm is an efficient algorithm for solving the inverse scattering problem for the system (\ref{2.1}).
\subsection{Case of skew-symmetric potential}
Consider now an important case for the system (\ref{1.1})--(\ref{1.2}), when
$$u_1(x,y)=u(x,y),\qquad u_2(x,y)=-\overline{u}(x,y),$$ where $u \in L_2(R^2) $.
This is a skew-symmetric potential. We have
$$U=\begin{pmatrix}0, &u_1\\u_2,&0\end{pmatrix}= \begin{pmatrix}0,&u\\-\overline{u},&0\end{pmatrix},\quad U ^*=-U.
$$
\begin{theorem}\label{3}
In the case of a two-dimensional Dirac system (\ref{2.1}) with a skew-symmetric potential, i.e. with the condition $u_1=u$,\,$u_2=-\bar{u}$,\,$u \in L_2(R^2)$, the scattering operator $S$ is unitary: $S^{-1}= S^*$. In this case, $F_{21}=G_{12}^*$ and $G_{21}=F_{12}^*$. Therefore, only one of the operators $F_{21}$ or $G_{21}$ can serve as scattering data.
\end{theorem}
\begin{proof}
Let $\psi(x,y)$ and $\varphi(x,y)$ be two admissible solutions for the Dirac system with skew-symmetric potentials $u_1=u$,\,$u_2=-\bar{u}$. Let for the solution $\psi(x,y)$ note the incident waves by $a=\col(a_1,a_2),$ and the scattered ones by $b=\col(b_1,b_2)$. For the solution $\varphi(x,y)$, we note similar waves by $\hat{a}$ and $\hat{b}$. Since the functions $\psi(x,y)$ and $\varphi(x,y)$ satisfy the equations $L\psi=0$ and $L\varphi=0$, it is easy to see that
\begin{equation}\label{2.4}
   \frac{\partial }{\partial x}[\psi_1\cdot\bar{\varphi}_1]+ \frac{\partial }{\partial y}[\psi_2\cdot\bar{\varphi}_2]= 0.
\end{equation}
Integrating the identity (\ref{2.4}) over $x$ and $y$ from $-\infty$ to $+\infty$, we get
\begin{equation}\label{2.5}
   (b_1,\hat{b}_1)_{L_2}- (a_1,\hat{a}_1)_{L_2}+ (b_2,\hat{b}_2)_{L_2}- (a_2,\hat {a}_2)_{L_2}=0.
\end{equation}
In other words, the inner products in $L_2(R^2,E^2)$ coincide: $(b,\hat{b})= (a,\hat{a}).$ Since $b=Sa$ and $ \hat{b}=Sa$, then $(Sa,S\hat{a})=(a,\hat{a})$ and therefore $SS^{*}=S^{*}S=I$ , that is, the operator $S$ is unitary.
\end{proof}

\begin{theorem}\label{4}
In the case of the
skew-symmetric potential $u_1=u$,\,$u_2=-\bar{u}$,\,$u \in L_2(R^2)$ in
 the Dirac system (\ref{2.1}), the necessary and sufficient conditions for scattering data operators are the condition $||F_{21}||<1$ or $||G_{21}||<1$.
\end{theorem}
\begin{proof}
Let us apply the criterion from Remark \ref{2.2} and Theorem \ref{2.1}. Since the conditions of  theorem  \ref{2.2} hold,  for the operator $G_{21}$ to be scattering data it is necessary and sufficient that for any real $-\infty \le \lambda \le +\infty$ there exists an operator $( I-Q_{\lambda}G_{21}G_{21}^{*}Q_{\lambda})^{-1}.$  The operator $Q_{\lambda}G_{21}G_{21}^{* }Q_{\lambda}$ is a self-adjoint integral operator with a Hilbert-Schmidt kernel, and its norm $||Q_{\lambda}G_{21}G_{21}^{*}Q_{\lambda}||$ is continuous depends on $\lambda$; moreover, it decreases monotonically as $\lambda \rightarrow \infty$. If $||G_{21}G_{21}^{*}|| \ge 1,$ then for some $\lambda=\lambda_0$ the norm of the operator $||Q_{\lambda_0}G_{21}G_{21}^{*}Q_{\lambda_0}||=1$, but then the operator $(I-Q_{\lambda_0}G_{21}G_{21}^{*}Q_{\lambda_0})^{-1}$ does not exist, since the norm of a self-adjoint positive operator coincides with the largest eigenvalue. Therefore $||G_{21}G_{21}^{*}|| <1,$ which is equivalent to $||G_{21}|| <$1. This condition is sufficient. The left factorization and the case of scattering data $F_{21}$ can be considered similarly.
\end{proof}
This implies the following statement.
\begin{remark}\label{2.3}
The operator $\mathfrak{A}$, which takes the potential $u \in L_2(R^2)$ to the scattering data, the function $f(\xi,\eta)$, which is the kernel of the integral operator $F$ of the scattering data with $|| F|| <$1. This is the description of the range of the operator setting IST.
\end{remark}
\begin{remark}
The relation between the potential $u(x,y)$ in the Dirac system and the scattering data $f(\xi,\eta)$ is effectively described by a chain of integral equations. Let us present the operator form for the transformer $\mathfrak{A}^{-1}$. Let the function $f(\xi,\eta)$ satisfy the condition that the integral operator $F$ for which the function $f(\xi,\eta)$ is a kernel has norm less than one. Consider an operator function of two real numbers
$\lambda$ and $\mu:$
\begin{equation}\label{2.6}
   U_{\lambda,\mu}=(I-F^*Q_{\lambda}FP_{\mu} )^{-1}F^*.
\end{equation}
The operator $U_{\lambda,\mu}$ is integral with kernel $u(\xi,\eta;\lambda,\mu )$, then
\begin{equation}\label{2.7}
   u(x,y)=u(y,x;x,y)=\mathfrak{A}^{-1}f.
\end{equation}
\end{remark}
\section{Evolution of scattering data}
\subsection{Equations for scattering data}
Let $u(x,y;t)$ be the potential for the Dirac system $L$ in the Lax representation (\ref{1.3}). This means that the function $u(x,y;t)$ is a solution to the system(\ref{1.1})--(\ref{1.2}). Let us consider the question of how the scattering data for such a potential change with time. If $\psi$ the solution of the Dirac system
$L\psi=0$, then it follows from the Lax representation that the function $\varphi=P\psi$ is also a solution of the Dirac system $L\varphi=0$. Of course, this requires that the function $\psi$ be sufficiently smooth in its arguments. From the equality $\varphi=P\psi$ and the explicit form of the operator $P$ given in (\ref{1.3}), we have
\begin{equation}\label{3.1}
   \varphi_1=(\mathcal{D}+v_1)\psi_1-2u_x\psi_2,\quad \varphi_2=-2\overline{u}_y\psi_1+(\mathcal{D}-v_2)\psi_2,
\end{equation}
where $\displaystyle \mathcal{D}=i\frac{\partial}{\partial t}-\frac{\partial^2}{\partial x^2}+\frac{\partial^2}{\partial y^2}$. We are interested in the $F_{21}$ scattering data. They relate the first incident wave of solutions $\psi$ and $\varphi$ to the second scattered wave of solutions to the Dirac system $L\psi=0$ and $L\varphi=0$. Note the incident waves for the $\psi$ solution by $(a_1,a_2)$, and the scattered waves by $(b_1,b_2)$. For the $\varphi$ solution, these waves will be noted by $(\hat{a}_1,\hat{a}_2)$ and $(\hat{b}_1,\hat{b}_2)$. Let $a_1(y)$ be a smooth finite function and the second wave $a_2\equiv 0$.  We define similar waves for the solution $\varphi$ using the relation (\ref{3.1}):
\begin{equation}\label{3.2}
   \hat{a}_1(y,t)=\lim_{x\rightarrow -\infty}\varphi_1(x,y;t)=\lim_{x\rightarrow -\infty}[(\mathcal{D}+ v_1)\psi_1(x,y;t)-2u_x\psi_2(x,y;t)]=\Bigl(\frac{\partial^2}{\partial y^2}+p_-(y,t) \Bigr)a_1(y),
\end{equation}
where $\displaystyle p_-(y,t)=\lim_{x\rightarrow -\infty}v_1(x,y;t).$
We can show that $\hat{a}_2=a_2=0$.
\begin{equation}\label{3.3}
\begin{array}{r}
   \hat{b}_2(x,t)=\lim_{y \rightarrow +\infty}\varphi_2(x,y;t)=\lim_{y\rightarrow +\infty}[-2\overline{u} _y\psi_1(x,y;t)+(\mathcal{D}-v_2)\psi_2(x,y;t)]=\\
  = \Bigl(i\frac{\partial}{\partial t}-\frac{\partial^2}{\partial x^2}-q_+(x,y)\Bigr)b_2(x,t).
   \end{array}
\end{equation}
Here $\displaystyle q_+(x,t)=\lim_{y\rightarrow +\infty}v_2(x,y;t)$. Since $\hat{b}_2=F_{21}(t)\hat{a}_1$,\,$b_2=F_{21}(t)a_1$, then the formula
\begin{equation}\label{3.4}
\Bigl( i\frac{\partial}{\partial t}-\frac{\partial^2}{\partial x^2}-q_+\Bigr)F_{21}a_1=F_{21}\Bigl( \frac{\partial^2}{\partial y^2}+p_-\Bigr)a_1.
\end{equation}
 follows from (\ref{3.2}),(\ref{3.3}).
Since $a_1$ is an arbitrary finite function, it follows from (\ref{3.4}) that
\begin{equation}\label{3.5}
  \Bigl (i\frac{\partial}{\partial t}-\frac{\partial^2}{\partial x^2}-\frac{\partial^2}{\partial y^2}-p_- -q_+\Bigr)f(x,y;t)=0,
\end{equation}
where $f(x,y;t)$ is the kernel of the integral $F_{21}$ scattering data operator $F$. The equation (\ref{3.5}) is a linear evolution equation where the real functions $p_-$ and $q_+$ (boundary data) serve as potentials. The equation (\ref{3.5}) can be given the form (\ref{1.8}).

\subsection{Analysis of solutions of a linear evolution equation for scattering data}
We consider the following a one-dimensional model.
\begin{theorem}\label{3.1}
A solution of the Cauchy problem for a linear equation
$$\Bigl[i\frac{\partial}{\partial t}+\varepsilon \frac{\partial^2}{\partial x^2}+w(x,t)\Bigr] u(x,t)=0,
$$
where $w(x,t)$ is a real, continuous, bounded function of its arguments in the space $L_2(R^1)$, and $\varepsilon=\pm1$ with initial conditions $u(x,t)|_{t =0}=u_0(x)\in L_2$ exists and is unique, and the operator $U(t)$ mapping the initial condition $u_0(x)$ into the solution $u(x,t)$ with $t>0 $, is a unitary operator in the space $L_2$.
\end{theorem}
\begin{proof}
The operator $\displaystyle A=\varepsilon\frac{\partial}{\partial x^2}$ is self-adjoint in the space $L_2$. The Cauchy problem is reduced by substituting $u=e^{-iAt}v$ to the equation $\displaystyle \frac{\partial v}{\partial t} =A(t)v$, where the operator $A(t)$ is bounded and depends continuously on $t$. For this equation, the solvability of the Cauchy problem follows from the corresponding integral equation. The unitary property of the operator $U(t)$ follows from the self-adjointness of the operator $A+w$ due to the fact that  the function $w(x,t)$ is real.
\end{proof}
\begin{theorem}\label{3.2}
 The Cauchy problem for a linear equation
\begin{equation}\label{3.6}
  \Bigl [i\frac{\partial}{\partial t}+\varepsilon_1 \frac{\partial^2}{\partial x^2}+\varepsilon_2\frac{\partial^2}{\partial y^2 }+p(y,t)+q(x,t)\Bigr]u(x,y;t)=0,
\end{equation}
where $\varepsilon_1=\pm 1,$\, $\varepsilon_2=\pm 1,$ the functions $p$ and $q$ are real, continuous, and bounded in their arguments, consists in finding a solution to this equation in the space $L_2( R^2)$ satisfying the initial conditions $u|_{t=0}=u_0(x,y)\in L_2(R^2)$. The solution of the Cauchy problem for the equation (\ref{3.6}) can be represented as following:
\begin{equation}\label{3.7}
   U(x,y;t)=U_1(t)U_2(t)U_0(x,y),
\end{equation}
where the semigroup unitary operators $U_1(t)$ act on the variable $x$ and the unitary operators $U_2(t)$ act on the variable $y$. These operators are constructed according to the theorem \ref{3.1} from the equations
$$
\begin{array}{c}
\displaystyle \Bigl [i\frac{\partial}{\partial t}+\varepsilon_1 \frac{\partial^2}{\partial x^2}+q(x,t)\Bigr]u(x,y ;t)=0,\\[2mm]
\displaystyle \Bigl[i\frac{\partial}{\partial t}+\varepsilon_2\frac{\partial^2}{\partial y^2}+p(y,t)\Bigr]u(x,y ;t)=0.
\end{array}
$$
If the equality (\ref{3.7}) contains the functions $U_0(x,y)$ and $U(x,y;t)$ as kernels of the integral operators $A_0$ and $A(t)$, then this equality admits representation
\begin{equation}\label{3.8}
   A(t)=U_1(t)A_0 JU_2^*(t) J,
\end{equation}
where $ J$ is an operator in the space $L_2$ that assigns to each function $a\in L_2$ its conjugate $J a=\bar{a}.$ It follows from the formula (\ref{3.8}) that $|| A(t)||=||A_0||$, that is, this norm does not depend on $t$.
\end{theorem}
\begin{proof}
The formula (\ref{3.7}) follows directly from  theorem \ref{3.1}. The formula (\ref{3.8}) follows from the representation of  action of the integral operator $U_0$ with kernel $u_0(x,y)$ on an arbitrary function $\varphi \in L_2$ in the form
$$(U_0\varphi)(x)=\int\,u_0(x,y)\varphi(y)\,dy=<u_0(x,\cdot),\, J\varphi>_{L_2},
$$
where $J \varphi=\bar{\varphi}$. Therefore, the operator $U_2(t)$ acting on the second variable of the function $u_0(x,y)$ can be transferred in the scalar product to $J \varphi$. The kernel $u_2(t)u_0(x,y)=\hat{u}_2(x,y)$ generates an integral operator $\hat{U}_2$ represented as:
$$\hat{U}_2\varphi=<\hat{u}_2(x,\cdot), J\varphi>_{L_2}=<u_0(x,\cdot), U_2^* J\varphi> _{L_2}=\int\,u_0(x,y) J U_2^*(t)J\varphi(y)\,dy.
$$
In other words, $\hat{U}_2=A_0JU_2^*(t)J.$ Acting on the last equality by the operator $U_1(t)$, we get the equality (\ref{3.8}). Since the operators $U_1(t)$,\,$U_2(t)$,\,$J$ are unitary, (\ref{3.8}) implies the equality $||A(t)||=||A_0|| $.
\end{proof}
\begin{remark}
If the initial data $u_0(x,y)$ are the product of functions depending only on $x$ and only on $y$, then the solution $u(x,y;t)$ has the same property.
\end{remark}
\begin{remark}
If the scattering data $f(x,y;t)$ for $t=0$ define an integral operator with the norm $||F(0)||<1$, then for any $t\ge 0$ the norm $|| F(t)||= ||F(0)||<1$. This follows from the fact that according to (\ref{3.5}) the function $f(x,y;t)$ satisfies the conditions of the theorem \ref{3.2}. If the scattering data $F(0)$ is an integral operator of rank 1, then for any $t\ge 0$ the scattering data integral operator $F(t)$ has rank 1.
\end{remark}
\section{Results on the inverse scattering problem for the Dirac system in the case when the integral operator of scattering data has rank 1}
 The following simple facts from the theory of integral operators are required in order to obtain explicitly all elements of the scattering operator and the potential from the scattering data,
\subsection{Constructive facts from the theory of Hilbert--Schmidt integral equations}
An integral operator $K$ with the Hilbert--Schmidt kernel $k(x,y)\in L^2(R^2)$ has the trace $\displaystyle \tr K=\int\,k(x,x)\, dx.$ An operator $K$ has rank 1 if its kernel $k(x,y)$ can be represented as a product of two functions, each of them depends on only one argument $k(x,y)=f(x)\cdot g(y)$, where $f,g\in L^2(R^1)$. If an integral operator $K_+$ can be represented as $$K_+\psi(x)=\int_{-\infty}^{x}\,k(x,y)\psi(y)\,dy,
$$
then this is the Volterra operator of positive polarity. If an integral operator $K_-$ can be represented as
$$K_-\psi(x) =\int_{x}^{\infty}\,k(x,y)\psi(y)\,dy,
$$
then this is the Volterra operator of negative polarity. If the kernel of  the Volterra operator can be represented as $k(x,y)=f(x)\cdot g(y)$, then we say that the Volterra operator has the kernel of rank 1. Note that the product of two integral operators of rank 1 is an integral operator of rank 1. The product of an integral operator of rank 1 and a Volterra operator with a kernel of rank 1 is an integral operator of rank 1. The product of two Volterra operators with kernels of rank 1 will no longer be an operator of rank 1, and will not even have a kernel of rank 1. We need such simple results.
\begin{lemma}\label{4.1}
Let $K$ be an integral operator with the Hilbert--Schmidt kernel of rank 1. Then $(I-K)^{-1}$ exists if and only if the trace $\tr K \neq 1$. In this case,
\begin{equation}\label{4.1}
(I-K)^{-1}=I+\frac{1}{1-\tr K}K.
\end{equation}
\end{lemma}
\begin{lemma}\label{4.2}
Let $K$ be a Hilbert--Schmidt integral operator of rank 1 with kernel $k(x,y)=f(x)\cdot g(y)$, and norm $||K||=||f||_{ L_2}\cdot||g||_{L_2}<1.$ Then the operator $I-K$ admits two-sided factorization into the Volterra integral operators:
\begin{equation}\label{4.2}
   I-K=(I+A_+)^{-1}(I+A_-)^{-1}=(I+B_-)^{-1}(I+B_+)^{-1},
\end{equation}
where the integral operators $A_{\pm}$,\,$B_{\pm}$ have kernels
$$A_{+}(x,y)=\frac{f(x)}{1-\tr K Q_x}g(y)\theta(x-y),\quad
A_{-}(x,y)=f(x)\frac{g(y)}{1-\tr K Q_y}\theta(y-x),
$$
\begin{equation}\label{4.3}
   B_{+}(x,y)=f(x)\frac{g(y)}{1- \tr K P_y}\theta(x-y),\quad
   B_{-}(x,y)=\frac{f(x)}{1-\tr K P_x}g(y)\theta(y-x),
\end{equation}
where
$$\tr K Q_x=\int_{-\infty}^{x}\,f(s)g(s)\,ds,
$$
$$\tr K P_x=\int_{x}^{\infty}\,f(s)g(s)\,ds.
$$
\end{lemma}
\begin{lemma}\label{4.3}
Let $K_+$ be the Hilbert--Schmidt Volterra integral operator of positive polarity with a kernel of rank 1 and $k(x,y)=f(x)\cdot g(y)$. Then there is an operator
\begin{equation}\label{4.4}
   (I-K_+)^{-1}=I+\hat{K}_+,
   \end{equation}
where $\hat{K}_+$ is the Volterra Hilbert--Schmidt integral operator of positive polarity with a kernel of rank 1, whose kernel $\hat{K}_+(x,y)$ has the form
\begin{equation}\label{4.5}
   \hat{K}_+(x,y)=\hat{f}(x)\hat{g}(y)\theta(x-y),
\end{equation}
where
\begin{equation}\label{4.6}
   \hat{f}(x)=f(x)e^{\tr K_+ Q_x},\quad \hat{g}(y)=g(y)e^{-\tr K_+ Q_y},
\end{equation}
where $Q_{\lambda}$ - projector
$$Q_{\lambda}\psi(x)=\theta(\lambda-x)\psi(x), \quad \tr( K_+ Q_x)=\int_{-\infty}^{x}\, f(s)g(s)\,ds.
$$
\end{lemma}
Lemma \ref{4.2} contains various Volterra operators $(I+A_+)^{-1}$ and $(I+B_+)^{-1}$. Let us  take a look at their explicit forms. Lemma \ref{4.2} implies
\begin{lemma}\label{4.4}
Let $K_+$ be a Hilbert--Schmidt Volterra integral operator with a kernel of rank 1 and a kernel of the form
\begin{equation}\label{4.7}
   K_+(x,y)=\frac{f(x)g(y)\theta(x-y)}{1-T(x)},\quad T(x)=\int_{-\infty}^{ x}\,f(s)g(s)\,ds.
\end{equation}
Then the operator
\begin{equation}\label{4.8}
    (I+K_+)^{-1}=I-\hat{K}_+, \quad \hat{K}_+(x,y)=\frac{f(x)g(y)\theta (x-y)}{1-T(y)}.
\end{equation}
\end{lemma}
In the simplest case, when we have $g(y)=\overline{f}(y)$  in  lemma \ref{4.4}  a nonlinear operator arises that  maps the function $f(x)$ into the function
$\displaystyle F(x)=\frac{f(x)}{1-\int_{-\infty}^{x}\,|f(s)|^2\,ds} $.
Such an operator has important properties. This type of nonlinear operators can play an important independent role.
\begin{lemma}\label{4.5}
In the space $L_2(R^1)$, the nonlinear operator $\mathcal{A}_k$ depending on the positive parameter $k<1$ and acting according to the formula
\begin{equation}\label{4.9}
    \mathcal{A}_k f(x)=\frac{\sqrt{1-k^2}f(x)}{1-\frac{k^2}{||f||_{L_2}^2 }\int_{-\infty}^{x}\,|f(s)|^2\,ds}.
  \end{equation}
  The operator $\mathcal{A}_k$ maps  each function $f\in L_2$,\,$||f||\neq0$ into the function $F(x)=\mathcal{A}_k f(x), $ \,$L_2$-- norm is equal to $||f||_{L_2}:$\,$||F||_{L_2}=||f||_{L_2}.$  The inverse operator $\mathcal{A}_k^{-1}$ has the form
\begin{equation}\label{4.10}
   \mathcal{A}_k^{-1}F(x)=\frac{\sqrt{1-k^2}F(x)}{1-\frac{k^2}{||F||_ {L_2}^2}\int_{x}^{+\infty}\,|F(s)|^2\,ds}.
\end{equation}
In the case of $||f||=0$ and $||F||=0$, we set $\mathcal{A}_k0=0$,\,$\mathcal{A}_k^{-1}0= 0$.
\end{lemma}
  As a special case of  lemma \ref{4.5} is the following statement.
\begin{lemma}\label{4.6}
A nonlinear operator $\mathcal{A}$ defined on functions $f\in L_2(R^1)$ with condition $||f||<1$ by equality
\begin{equation}\label{4.11}
    \mathcal{A}f(x)=\frac{f(x)\sqrt{1-||f||^2}}{1-\int_{-\infty}^{x}\,|f( s)|^2\,ds}
\end{equation}
preserves the norm $||\mathcal{A}f||=||f||.$ The inverse operator $\mathcal{A}^{-1}$ has the form
\begin{equation}\label{4.12}
      \mathcal{A}^{-1}g(y)=\frac{g(y)\sqrt{1-||g||^2}}{1-\int_{x}^{\infty}\,|g(s)|^2\,ds}.
\end{equation}
\end{lemma}
The proof of all lemmas is reduced to a simple verification of constructive assertions.
\subsection{Explicit form of the scattering operator and potential}
\begin{theorem}\label{4.1}
Let the scattering operator $S=||F_{ij}||_{i,j=1}^{2}$ be unitary in the case of a skew-symmetric potential in the Dirac system. Let the scattering data of integral operator $F_{21}$ have rank 1 and its kernel $f(x,y)=kf(x)g(y)$, where the functions $f, g \in L_2$,\,$ ||f||=||g||=1$ ,\,$k<1$. Then all elements $F_{ij}$ of the scattering operator are Hilbert--Schmidt integral operators, and their kernels $F_{ij}(x,y)$ are explicitly expressed in terms of $f(x,y)$ as
\begin{equation}\label{4.13}
\begin{array}{l}
\displaystyle F_{11}(x,y)= \frac{-k^2 \overline{g(x)}g(y)}{1-k^2\int_{x}^{\infty}\, |g(s)|^2\,ds}=\frac{-k^2}{\sqrt{1-k^2}}\overline{\mathcal{A}_k^{-1}g(x) }\cdot g(y),\\[3mm]
   \displaystyle F_{22}(x,y)= \frac{-k^2 f(x)\overline{f(y)}}{1-k^2\int_{-\infty}^{x}\ ,|f(s)|^2\,ds}=\frac{-k^2}{\sqrt{1-k^2}} f(x) \cdot \overline{\mathcal{A}_k f( y)}\\[4mm]
   \displaystyle F_{12}(x,y)= -k \overline{\mathcal{A}_k^{-1}g(x)} \cdot \overline{\mathcal{A}_k f(y)},
\end{array}
\end{equation}
where $$\mathcal{A}_kf=\frac{\sqrt{1-k^2}f(x)}{1-k^2\int_{-\infty}^{x}\,|f(s )|^2\,ds} ,\quad
\mathcal{A}_k^{-1}g(x)=\frac{\sqrt{1-k^2}g(x)}{1-k^2\int_{x}^{\infty}\ ,|g(s)|^2\,ds}.$$
\end{theorem}
\begin{proof}
By virtue of  equation (\ref{2.3}) in Remark \ref{2.1} and  above lemmas \ref{4.2}-\ref{4.4}, we have  (\ref{4.13}).
\end{proof}
\begin{remark}\label{4.8}
As follows from  theorem \ref{4.1}, if the scattering data $F_{21}$ is a Hilbert--Schmidt integral operator with the norm $||F_{21}||=k<1,$ then the other scattering data $G_ {21}=F_{12}^*$ are Hilbert--Schmidt integral operators of rank 1 with the norm $||G_{21}||=||F_{12}||=k.$
\end{remark}
\begin{theorem}\label{4.2}
Let all the conditions of  theorem \ref{4.1} be satisfied. Then the potential $u(x,y)$ is expressed in terms of the scattering data $f(x,y)$ as
\begin{equation}\label{4.14}
   u(x,y)=\frac{\overline{f(x,y)}}{1-\int_{-\infty}^{x}\,d\xi \int_{y}^{\infty} \,d\eta|f(\xi,\eta)|^2}.
\end{equation}
\end{theorem}
\begin{proof}
For the potential $u$ there are formulas (\ref{2.6})--(\ref{2.7}). Let us apply Lemma \ref{4.1} to the operator $(I-F^*Q_{\lambda}FP_{\mu})^{-1}$. For this, we find that
$$\tr (F^*Q_{\lambda}FP_{\mu})= \int_{-\infty}^{\lambda}\,d\xi \int_{\mu}^{\infty}\, d\eta |f(\xi,\eta)|^2.
$$
Then the representation (\ref{4.14}) follows  from (\ref{2.6}) due to (\ref{2.7}).
\end{proof}

\section{Explicit Solutions with Scattering Data of Rank 1}
If $u(x,y;t)$ is a solution of the equation (\ref{1.1}) with pseudopotentials
$v_1$,\,$v_2$ representable as (\ref{1.5}) with boundary data
$p_-(y,t)$,\,$q_+(x,t)$, which are real, continuous, and bounded functions of their arguments, then the IST of such a solution is an integral operator $F(t)$ with the Hilbert--Schmidt  kernel $f( x,y;t)$ , and this kernel satisfies the linear differential equation (\ref{1.8}). Here norm $||F(t)||= ||F(0)||<1$. If the integral operator $F(t)$ has rank 1, then $f(x,y;t)=f(x,t)\cdot g(y,t).$ Moreover, $||F(t)| |=||f||_{L_2}\cdot||g||_{L_2}.$ The solution $u(x,y;t)$, as the potential of the two-dimensional Dirac system, is uniquely and explicitly expressed in terms of the scattering data $ F(t)$ in the form:
\begin{equation}\label{5.1}
   u(x,y;t)=\frac{\overline{f}(x,t)\cdot\overline{g}(y,t)}{1-F(x)G(y)},
\end{equation}
where
\begin{equation}\label{5.2}
F(x)=\int_{-\infty}^{x}\,|f(s,t)|^2\,ds, \quad G(y)=\int_{y}^{+\infty} \,|g(s,t)|^2\,ds.
\end{equation}
The opposite is also true:
\begin{theorem}
Let $ f(x,t)$ and $ g(y,t)$ be solutions of linear equations
\begin{equation}\label{5.3}
   \begin{array}{c}
     \Bigl(i\frac{\partial}{\partial t} + \frac{\partial^2}{\partial x^2}\Bigr)\overline{f}(x,t)+q_+(x, t)\overline{f} (x,t)=0,\\
      \Bigl( i\frac{\partial}{\partial t} + \frac{\partial^2}{\partial y^2} \Bigr)\overline{g}(y,t)+p_-(y, t)\overline{g} (y,t)=0,
   \end{array}
\end{equation}
where $f(x,0)$,\,$ g(y,0)$ are arbitrary functions from $L_2$ such that
$||f(x,0)||_{L_2}\cdot||g(y,0)||_{L_2}<1,$ and functions $p_-(y,t)$,\,$ q_+(x,t)$ are real, continuous, bounded functions of their arguments. Then the formula (\ref{5.1}) is a solution of the equation (\ref{1.1}) for $t>0$ the representations (\ref{1.5}) are satisfied, and for $t=0$ the solution $u(x,y; t)=u(x,y;0)$.
\end{theorem}
\begin{proof}
The correctness of this statement follows from the direct substitution of $u$ of the form (\ref{5.1}) into the equation (\ref{1.1}). We only note that, in view of the accepted conditions
$$
\begin{array}{c}
\displaystyle i\frac{\partial}{\partial t}F(x)=\overline{f}(x,t)\cdot f_x'(x,t)-\overline{f}_x'(x,t )\cdot f(x,t),\\ [2mm]
\displaystyle i\frac{\partial}{\partial t} G(y)=\overline{g}_y'(y,t)\cdot g(y,t)-\overline{g}(y,t) \cdot g_y'(y,t), \\ [2mm]
  \displaystyle v_1+v_2=p_-(y,t)+q_+(x,t)+2 \Bigl[\frac{F (|g|^2)_y'-G(|g|^2)_x' }{1-F\cdot G}-\frac{G^2|f|^4+F^2|g|^4}{(1-F\cdot G)^2} \Bigr].
\end{array}
$$
The other calculations are very simple.
\end{proof}

\begin{remark}\label{5.1}
Since the pseudopotentials $v_1$ and $v_2$ are real, the equation (\ref{1.1}) has the integral of motion
$$\int\int|u(x,y;t)|^2\,dxdy=||u||_{G-Sch}^2=C,
$$
independent of $t$. For the solution (\ref{5.1}), this $C$ constant is:
\begin{equation}\label{5.8}
C=||u||_{G-Sch}^2=|\ln(1-k^2)||,\quad k=||f(x,0)||\cdot||g(y,0)||.
\end{equation}
\end{remark}
Let us give an example of solving linear equations (\ref{5.1}). Consider the Cauchy problem for such an equation
\begin{equation}\label{5.9}
  i \frac{\partial p(x,t)}{\partial t}+\frac{\partial^2 p(x,t)}{\partial x^2}=0,\quad p(x,0 )=\Bigl( \frac{2}{\pi}\Bigr)^{\frac{1}{4}}e^{-x^2},\quad ||p||_{L_2}=1 .
\end{equation}
Passing in the problem (\ref{5.9}) to the Fourier transform:
$$
\tilde{p}(\lambda,t)=\int\,e^{i\lambda x}p(x,t)\,dx,
$$
we get
\begin{equation}\label{5.10}
    i \frac{\partial \tilde{p}(\lambda, t)}{\partial t}-\lambda^2 \tilde{p}(\lambda,t)=0,\quad \tilde{p}( \lambda, 0)=(2\pi)^{\frac{1}{4}}e^{-\frac{\lambda^2}{4}}.
\end{equation}
The solution of the problem (\ref{5.10}) has the form
\begin{equation}\label{5.11}
   \tilde{p}(\lambda,t)=e^{-i\lambda^2 t} \tilde{p}(\lambda,0)=(2\pi)^{\frac{1}{4} }e^{-\lambda^2(\frac{1}{4}+it)}.
\end{equation}
The inverse Fourier transform gives
\begin{equation}\label{5.12}
    p(x,t)=\Bigl( \frac{2}{\pi}\Bigr)^{\frac{1}{4}}\frac{1}{\sqrt{1+4it}}e^{ -\frac{x^2}{1+4it}}.
\end{equation}
Note also that the function
$$\mathcal{P}(x,t)=\int_{-\infty}^{x}\,|p(s,t)|^2\,ds=\frac{1}{2}(1 +\Phi\Bigl( \frac{x\sqrt{2}}{1+16t^2}\Bigr),
$$
where $\Phi(x)=\frac{2}{\sqrt{\pi}}\int_0^x\,e^{-s^2}\,ds $ is the probability integral (the Fresnel integral). Then the solution  of the equations (\ref{1.1}) has the explicit form
\begin{equation}\label{5.13}
   u(x,y;t)=\frac{kp(x,t)p(y,t)}{1-k^2\mathcal{P}(x,t)\mathcal{\hat{P}} (y,t)},
\end{equation}
where $\mathcal{\hat{P}}(y,t)=1- \mathcal{P}(y,t)$.

{\bf Acknowledgments.}\\
 The author thanks the University of Bonn for the 1.5 month invitation in 2019, which allowed him to return to research on the soliton theory. For a number of reasons, a while ago, he had to abandon work on this subject. He is  grateful to Professor Herbert Koch for stimulating discussions, correspondence, and citing important published work. He is also grateful to the Simons Foundation for the financial support of scientists from the Institute of Mathematics of the National Academy of Sciences of Ukraine, which made it possible to write this work.

\end{document}